\documentclass[12pt,a4paper]{amsart}

\newtheorem{Result}{Result}
\newtheorem{Thm}  [Result]{Theorem}
\newtheorem{Prop} [Result]{Proposition}
\newtheorem{Lemma}[Result]{Lemma}

\newtheorem*{Def}{Definition}

\def\N{{\mathbb{N}}}
\let\eps\varepsilon

\let\maparrow\longrightarrow
\def\map#1#2#3{#1\colon#2\,{\maparrow}\,#3}

\def\ptensor{\mathop{\hat\otimes_\pi}}
\def\itensor{\mathop{\hat\otimes_\eps}}

\begin{document}

\title{On subprojectivity of $C(K,X)$}

\author{Manuel Gonz\'alez}
\address{Departamento de Matem\'aticas, Facultad de Ciencias,
  Universidad de Cantabria, \hbox{E-39071} Santander, Spain}

\author{Javier Pello}
\address{Escuela Superior de Ciencias Experimentales y Tecnolog\'\i a,
  Universidad Rey Juan Carlos, \hbox{E-28933} M\'ostoles, Spain}

\subjclass[2010]{46B03}
\keywords{Banach space; subprojective space}

\begin{abstract}
We show that the Banach space $C(K,X)$ is subprojective
if $K$ is scattered and $X$ is subprojective.
\end{abstract}

\maketitle


A Banach space~$X$ is called \emph{subprojective} if every closed
infinite-dimensional subspace of~$X$ contains an infinite-dimensional
subspace that is complemented in~$X$.
Subprojective spaces were introduced by Whitley~\cite{whitley} to give
conditions for the conjugate of an operator to be strictly singular.
They are relevant in the study of the perturbation classes problem for
semi-Fredholm operators, which has a negative answer in
general~\cite{gonzalez} but a positive answer when one of the spaces is
subprojective \cite{gonzalez-et-al}. Recently Oikhberg and Spinu made a
systematic study of subprojective spaces \cite{oikhberg-spinu}, widely
increasing the family of known examples of spaces in this class.

Here we prove that
$C(K,X)$ is subprojective whenever $K$ is a scattered compact and
$X$ is a subprojective Banach space;
a compact space~$K$ is said to be \emph{scattered} or \emph{dispersed}
if every non-empty closed subset of~$K$ contains an isolated point.
In the case where $K = [0,\lambda]$ for some ordinal~$\lambda$,
this result was previously obtained for countable~$\lambda$
\cite[Theorem 4.1]{oikhberg-spinu} and later for arbitrary~$\lambda$
\cite[Theorem 2.9]{galego-MG-JP}.

This generalises the scalar case, where it was already known that
$C(K)$ is subprojective if and only if $K$ is scattered
\cite[Theorem~11]{lotz-peck-porta} \cite[Main Theorem]{pelczynski-semadeni},
and fully characterises the subprojectivity of~$C(K,X)$, as
the subprojectivity of $C(K,X)$ implies that
of $C(K)$ and~$X$ \cite[Proposition 2.1]{oikhberg-spinu}.

We will use standard notation. 
If $(x_n)_{n\in\N}$ is a sequence in~$X$, then $[x_n:n\in\N]$ will denote
the closed linear span of $(x_n)_{n\in\N}$ in~$X$.
Given a (bounded, linear) operator $\map TXY$, 
$N(T)$ and~$R(T)$ denote the kernel and the range of~$T$, respectively.
An operator $\map TXY$ is strictly singular if $T|_M$ is an isomorphism only
if $M$ is finite-dimensional.

The injective tensor product of $X$ and~$Y$ is denoted by $X\itensor Y$.
Note that $C(K,X)$ can be identified with $C(K) \itensor X$
\cite[Section 3.2]{ryan}, and they will be used interchangeably in the sequel.
For countable ordinals $\lambda$, $\mu$ it was proved in \cite{galego-samuel}
that the projective tensor product $C([0,\lambda]) \ptensor C([0,\mu])$ is
subprojective.



Let $K$, $L$ be compact spaces and let $\map\varphi KL$ be a continuous
function. It is well known that $\varphi$ induces an operator $C(L) \maparrow
C(K)$ that maps each $f \in C(L)$ to $f\circ\varphi \in C(K)$; we will
denote this operator by $\map{\tilde\varphi}{C(L)}{C(K)}$.

\begin{Def}
Let $X$ be a Banach space and let $Z$ be a subspace of~$X$.
We will say that $Z$ is subprojective with respect to~$X$ if
every closed infinite-dimensional subspace of~$Z$ contains an
infinite-dimensional subspace complemented in~$X$.
\end{Def}

Note that this is a stronger notion for~$Z$ than merely being subprojective,
as it requires the subspace to be complemented in~$X$ and not just in~$Z$.
Also, a space~$X$ is subprojective if and only if each of its subspaces
is subprojective with respect to~$X$.

\begin{Prop}
\label{prev}
Let $X$ be a Banach space, let $\map PXX$ be a projection such that
$R(P)$ is subprojective and let $Z$ be a closed subspace of~$X$ such that
$P(Z) \subseteq Z$ and $Z \cap N(P)$ is subprojective with respect to~$X$.
Then $Z$ is subprojective with respect to~$X$.
\end{Prop}

\begin{proof}
Let $M$ be a closed infinite-dimensional subspace of~$Z$.
If $M \cap N(P)$ is infinite-dimensional, then it contains
another infinite-dimensional subspace complemented in~$X$ by hypothesis.

Otherwise, if $M \cap N(P)$ is finite-dimensional, we can assume that
$M \cap N(P) = 0$ by passing to a further subspace of~$M$ if necessary.
If $P(M)$ is closed, then $P|_M$ is an isomorphism and,
if $N$ is an infinite-dimensional subspace of~$M$ such that $R(P) =
P(N) \oplus H$ for some closed subspace~$H$, then $X = N \oplus P^{-1}(H)$.

We are left with the case where $M \cap N(P) = 0$ and $P(M)$ is not closed.
Take a normalised sequence $(x_n)_{n\in\N}$ in~$M$ such that $\|P(x_n)\| <
2^{-n}$ for every $n\in\N$. Since any weak cluster point of $(x_n)_{n\in\N}$
must be in $M \cap N(P) = 0$, by passing to a subsequence
\cite[Theorem~1.5.6]{kalton-albiac} we can assume that $(x_n)_{n\in\N}$
is a basic sequence and that there exists a bounded sequence
$(x^*_n)_{n\in\N}$ in~$X^*$ such that $x^*_i(x_j) = \delta_{ij}$ for every
$i$, $j\in\N$ and $\sum_{n=1}^\infty \|x^*_n\| \|P(x_n)\| < 1$. Then $K(x) =
\sum_{n=1}^\infty x^*_n(x) P(x_n)$ defines an operator $\map KXX$ with
$\|K\| < 1$ that maps $K(x_n) = P(x_n)$ for every $n\in\N$, and then
$I - K$ is an automorphism on~$X$ such that $[ (I{-}K)(x_n) : n\in\N ] =
[ (I{-}P)(x_n) : n\in\N ] \subseteq Z \cap N(P)$ and, if $N$ is an
infinite-dimensional subspace of $[ (I{-}K)(x_n) : n\in\N ]$ complemented
in~$X$, then $(I{-}K)^{-1}(N) \subseteq M$ and is still complemented in~$X$.
\end{proof}

\begin{Prop}
\label{crux}
Let $X$ be a subprojective Banach space, let $\lambda$ be an ordinal,
let $K$ be a sequentially compact space and let $\map\varphi K{[0,\lambda]}$
be a continuous surjection. Then $\tilde\varphi(C([0,\lambda])) \itensor X$
is subprojective with respect to $C(K,X)$.
\end{Prop}

Here we are using the identification $C(K,X) \equiv C(K) \itensor X$ and
the fact that, if $X$ and $Y$ are Banach spaces and $M$ is a closed subspace
of~$X$ and $N$ is a closed subspace of~$Y$, then $M \itensor N$ can be seen
as a subspace of $X \itensor Y$ \cite[Comments after Proposition 3.2]{ryan}.

\begin{proof}
We will proceed by induction in~$\lambda$. The result is trivial for
$\lambda = 0$, as then $\tilde\varphi(C([0,\lambda])) \itensor X$ is the
set of constant functions, which is complemented in~$C(K,X)$ and isomorphic
to~$X$, which is subprojective.

Let us then assume that the result is true for every continuous
surjection $K\maparrow[0,\mu]$ with $\mu < \lambda$.
Consider, for each ordinal $\mu < \lambda$, the set $K_\mu =
\varphi^{-1}([0,\mu]) \subseteq K$, which is both open and closed,
and the operator $\map{P_\mu}{C(K,X)}{C(K,X)}$ given by $P_\mu(f) =
f \chi_{K_\mu}$, which is a projection with range isometric to $C(K_\mu,X)$;
also note that
$P_\mu(f) \mathrel{\setbox0\hbox{$\textstyle\longrightarrow$}
  \dimen0 \ht0 \advance\dimen0 -\fontdimen22\textfont2
  \textstyle\mathop{\vcenter to 2\dimen0{\box0\vss}}\limits_\mu} f$
for every $f\in\tilde\varphi(C_0([0,\lambda])) \itensor X$.

We will first prove that $\tilde\varphi(C_0([0,\lambda])) \itensor X$
is subprojective with respect to $C(K,X)$.
Let $M$ be a closed infinite-dimensional subspace of
$\tilde\varphi(C_0([0,\lambda])) \itensor X$.
If there exists some $\mu < \lambda$ for which $P_\mu|_M$ is not strictly
singular, we can assume that $P_\mu|_M$ is an isomorphism by passing
to a further subspace of~$M$ if necessary and then $P_\mu(M)$, seen as a
subspace of $C(K_\mu,X)$, contains an
infinite-dimensional subspace complemented in $C(K_\mu,X)$ by the induction
hypothesis (with $\map{\varphi|_{K_\mu}}{K_\mu}{[0,\mu]}$) so $M$ contains
an infinite-dimensional subspace complemented in~$C(K,X)$
\cite[Proposition 2.3]{oikhberg-spinu}.
This must necessarily be the case if $\lambda$ is not a limit ordinal,
as then there exists some ordinal~$\mu$ such that $\lambda = \mu + 1$,
for which $P_\mu$ is the identity on~$M$ because functions in~$M$
vanish at $\varphi^{-1}(\lambda)$.

Otherwise, if $M \subseteq \tilde\varphi(C_0([0,\lambda])) \itensor X$
but $P_\mu$ is strictly singular for every $\mu < \lambda$, then $\lambda$
must be a limit ordinal by the previous sentence. Also, for every
$\mu < \lambda$ and $\eps > 0$, there exists $f \in M$ such that
$\|f\| = 1$ and $\|P_\mu(f)\| < \eps$, and then there is $\nu > \mu$
such that $\|P_\nu(f) - f\| < \eps$.
By induction, starting with an arbitrary $\mu_1 < \lambda$,
there exists a strictly increasing sequence of ordinals $\mu_1 < \mu_2 <
\cdots < \lambda$ and a sequence $(f_n)_{n\in\N}$ of normalised functions
in~$M$ such that $\|P_{\mu_n}(f_n)\| < 2^{-n} / 32$ and
$\|P_{\mu_{n+1}}(f_n) - f_n\| < 2^{-n} / 32$ for every $n\in\N$.

Now, for each $n\in\N$, there exist $t_n \in K$ such that $\|f_n(t_n)\| = 1$
and then a normalised $x^*_n \in X^*$ such that $x^*_n(f_n(t_n)) = 1$;
note that $t_n \in K_{\mu_{n+1}} \setminus K_{\mu_n}$ because $f_n$ cannot
attain its norm outside of $K_{\mu_{n+1}} \setminus K_{\mu_n}$.
As $K$ is sequentially compact, by passing to a subsequence, we may assume
that $(t_n)_{n\in\N}$ converges to some $t_\infty \in K$. Let
$\map Q{C(K,X)}{c_0}$ and $\map J{c_0}{C(K,X)}$ be the operators defined as
$$Q(f) = \bigl( x^*_n \bigl( f(t_n) - f(t_\infty) \bigr) \bigr)_{n\in\N},$$
$$J \bigl( (\alpha_n)_{n\in\N} \bigr) =
  \sum_{n=1}^\infty \alpha_n (P_{\mu_{n+1}} - P_{\mu_n})(f_n) =
  \sum_{n=1}^\infty \alpha_n f_n \chi_{K_{\mu_{n+1}} \setminus K_{\mu_n}};$$
then $Q$ and~$J$ are well defined, $\|Q\| = 2$ and $J$ is an isometry
into $C(K,X)$ (into $\tilde\varphi(C_0([0,\lambda])) \itensor X$, actually),
and $QJ$ is the identity on~$c_0$, so $JQ$ is a projection in $C(K,X)$
with range isometric to~$c_0$. And, since
\begin{eqnarray*}
  \sum_{n=1}^\infty \| (P_{\mu_{n+1}} - P_{\mu_n})(f_n) - f_n \|
    &\leq& \sum_{n=1}^\infty \bigl( \|P_{\mu_{n+1}}(f_n) - f_n\|
        + \|P_{\mu_n}(f_n)\| \bigr) \\
    &<& \sum_{n=1}^\infty 2^{-n}/16 = 1/16,
\end{eqnarray*}
it follows that $[f_n : n\in\N]$ is also isomorphic to~$c_0$ and complemented
in $C(K,X)$
\cite[Proposition 1.a.9]{lt1}.

Finally, let $Z = \tilde\varphi(C([0,\lambda])) \itensor X$, which is a
closed subspace of $C(K,X)$, fix $t_0 \in \varphi^{-1}(\lambda) \subseteq K$
and consider the natural projection $\map{P}{C(K,X)}{C(K,X)}$ defined as
$P(f) = 1 \otimes f(t_0) \in C(K) \itensor X$; then $R(P)$ is isomorphic
to~$X$, which is subprojective by hypothesis, and $R(P) \subseteq Z$.
Moreover, given $f \in Z \cap N(P)$, it holds that $f(t_0) = 0$ and
$f \in \tilde\varphi(C([0,\lambda])) \itensor X$, so $f$ must be constant
over $\varphi^{-1}(\lambda)$ and then $f|_{\varphi^{-1}(\lambda)} = 0$.
This means that  $Z \cap N(P) = \tilde\varphi(C_0([0,\lambda])) \itensor X$,
which is subprojective with respect to $C(K,X)$. Applying
Proposition~\ref{prev}, $Z = \tilde\varphi(C([0,\lambda])) \itensor X$
is subprojective with respect to $C(K,X)$.
\end{proof}

\begin{Lemma}
\label{lemma}
Let $X$, $Y$ be Banach spaces and let $M$ be a closed separable subspace of
$X \itensor Y$. Then there exist closed separable subspaces $M_X \subseteq X$
and $M_Y \subseteq Y$ such that $M \subseteq M_X \itensor M_Y$ as a
subspace of $X \itensor Y$.
\end{Lemma}

\begin{proof}
$M$ is separable, so there exists $(z_m)_{m\in\N}$ in~$M$ such that
$M = [ z_m : m\in\N ]$ and then, for each $m\in\N$, there exist
$(x_{m,n})_{n\in\N}$ in~$X$ and $(y_{m,n})_{n\in\N}$ in~$Y$ such that
$z_m \in [ x_{m,n} \otimes y_{m,n} : n\in\N ]$, so $M \subseteq
[ x_{m,n} \otimes y_{m,n} : m, n\in\N ]$. Let $M_X = [ x_{m,n} : m, n\in\N ]
\subseteq X$ and $M_Y = [ y_{m,n} : m, n\in\N ] \subseteq Y$; then $M_X$ and
$M_Y$ are separable and $M \subseteq M_X \itensor M_Y$, which can be seen as
a subspace of $X \itensor Y$ \cite[Proposition 3.2]{ryan}.
\end{proof}

\begin{Thm}
Let $K$ be a scattered compact space and let $X$ be a subprojective Banach
space. Then $C(K,X)$ is subprojective.
\end{Thm}

\begin{proof}
Let $M$ be a closed infinite-dimensional subspace of~$C(K,X)$, which we can
assume to be separable without loss of generality.
By Lemma~\ref{lemma}, there exist separable subspaces $G \subseteq C(K)$
and $Z \subseteq X$ such that $M \subseteq G \itensor Z \subseteq
C(K) \itensor X \equiv C(K,X)$. Without loss of generality, we can replace
$G$ with the least closed self-adjoint subalgebra with unit of~$C(K)$ that
contains it, as this is still separable, and then there exists a compact
space~$L$ and a continuous surjection $\map\varphi KL$ such that
$G = \tilde\varphi(C(L))$ 
\cite{eidelheit} \cite[Theorem 7.5.2]{semadeni},
so $C(L)$ is isomorphic to~$G$, which is separable, and this means in turn
that $L$ is metrisable.
Under these conditions, $L$ is scattered \cite[Lemma 2.5.1]{lacey}
and so homeomorphic to $[0,\lambda]$ for some countable ordinal~$\lambda$
\cite[Corollary 2.5.2]{lacey}.
By Proposition~\ref{crux}, $M$ contains an infinite-dimensional subspace
complemented in $C(K,X)$.
\end{proof}


\end{document}